\documentclass[a4paper, 11pt]{amsart}
\usepackage[latin1]{inputenc}
\usepackage[T1]{fontenc}
\usepackage[english]{babel}
\usepackage{amssymb}
\usepackage{amsmath}
\usepackage{amsthm}
\usepackage{amscd}
\usepackage{amsfonts}
\usepackage{stmaryrd}
\usepackage{pb-diagram}
\usepackage{epic,eepic,epsfig}
\usepackage{a4wide}
\usepackage{nextpage}
\usepackage{fancyhdr}
\usepackage{tikz}
\usepackage{verbatim}
\usepackage{comment}
\usepackage{mathrsfs}
\usepackage{xypic}

\newtheorem{prop}{Proposition}[section]
\newtheorem{cor}[prop]{Corollary}

\newtheorem{lem}[prop]{Lemma}
\newtheorem{thm}[prop]{Theorem}

\newtheorem{Proposition}[prop]{Proposition}

\theoremstyle{definition}

\newtheorem{defin}[prop]{Definition}

\newtheorem{exa}[prop]{Example}

\newtheorem{convention}[prop]{Convention}

\newcommand{\hp} {\mathop{\mathrm{h}_{\mathbb{P}^N}}}
\newcommand{\hF} {\mathop{\mathrm{h}_{\mathrm{F}}}}
\newcommand{\Jac} {\mathop{\mathrm{Jac}}}

\newcommand{\degr} {\mathop{\mathrm{deg}}}

\newcommand{\Spec} {\mathop{\mathrm{Spec}}}

\newcommand{\Q} {\mathbb{Q}}
\newcommand{\Z} {\mathbb{Z}}
\newcommand{\N} {\mathbb{N}}
\newcommand{\R} {\mathbb{R}}
\newcommand{\C} {\mathbb{C}}
\newcommand{\Qf} {\mathcal{Q}}
\newcommand{\m} {m}
\newcommand{\Oseen} {\mathcal{O}}
\newcommand{\Qbar} {\overline{\Q}}
\newcommand{\czero} {c_1}
\newcommand{\cone} {c_2}
\newcommand{\ctwo} {c_3}
\newcommand{\cthree} {c_{6}}
\newcommand{\cfour} {c_7}
\newcommand{\cfive} {c_8}
\newcommand{\csix} {c_9}
\newcommand{\cseven} {c_{13}}
\newcommand{\ceight} {c_{14}}
\newcommand{\cnine} {c_{15}}
\newcommand{\cten} {c_{16}}

\newcommand{\cthirteen} {c_{18}}
\newcommand{\cfourteen} {c_{19}}

\newcommand{\ctwenty} {c_{10}}

\newcommand{\ctwentyseven} {c_{11}}
\newcommand{\ctwentyeight} {c_{17}}

\newcommand{\cthirtyone} {c_{12}}
\newcommand{\cthirtytwo} {c_{12}}

\newcommand{\constii} {c_{4}}
\newcommand{\constiii} {c_{5}}

\newcommand{\AK}{{SS\Theta}}

\begin{document}

\title{Bertini and Northcott}
\author{{F}abien {P}azuki and {M}artin {Widmer}}
\address{Fabien Pazuki. University of Copenhagen, Institute of Mathematics, Universitetsparken 5, 2100 Copenhagen, Denmark, and Universit\'e de Bordeaux, IMB, 351, cours de la Lib\'eration, 33400 Talence, France.}
\email{fpazuki@math.ku.dk}
\address{Martin Widmer. Department of Mathematics, Royal Holloway, University of London,
Egham, Surrey, TW20 0EX, United Kingdom}
\email{Martin.Widmer@rhul.ac.uk}

\maketitle

\begin{abstract}
We prove a new Bertini-type Theorem with explicit control of the genus, degree, height, and the field of definition of the constructed curve. 
As a consequence we provide a general strategy to reduce certain height and rank estimates on abelian varieties over a number field $K$ to the case of jacobian varieties defined over a suitable extension of $K$. 
\end{abstract}

{\flushleft
}

\begin{center}
---------
\end{center}

\begin{center}
\textbf{Bertini et Northcott}
\end{center}

\begin{abstract}
Nous pr\'esentons un th\'eor\`eme de Bertini avec contr\^ole sur le genre, le degr\'e, la hauteur et le corps de d\'efinition de la courbe obtenue. Ceci fournit une strat\'egie de r\'eduction de certains \'enonc\'es d'estimations sur la hauteur ou le rang des vari\'et\'es ab\'eliennes g\'en\'erales d\'efinies sur un corps de nombres $K$ au cas des jacobiennes d\'efinies sur une extension de $K$ appropri\'ee.
\end{abstract}

{\flushleft
\textbf{Mots-Clefs:} Bertini, Northcott, hauteur, vari\'et\'es ab\'eliennes.\\
}

\begin{center}
---------
\end{center}

\thispagestyle{empty}

\section{Introduction}

Let $A$ be an abelian variety of dimension $g$ over a number field $K$. We denote 
by $\Delta_K$ the discriminant of $K$. Let $\mathrm{rank}(A(K))$ denote the Mordell-Weil rank of $A(K)$ and $\hF(A/K)$ the relative Faltings height of $A$. The first author recently gave in \cite{Paz16} a proof of the inequality 
\begin{alignat}1\label{heightrank}
\mathrm{rank}(A(K))\leq c(g) [K:\Q]^3 \max\{1, \hF(A/K), \log\vert \Delta_K\vert\},
\end{alignat}
for an explicit constant $c(g)>0$. 
It is unknown whether a dependence on $\Delta_K$ is needed.
However, (\ref{heightrank}) was achieved by combining two inequalities: the first one is a classical descent inequality between the Mordell-Weil rank of $A(K)$ and the logarithm of the product of the norms of the primes of bad reduction for $A/K$. The second one is an inequality between the logarithm of the product of the norms of the primes of bad reduction for $A/K$ and the Faltings height of $A$, obtained using the strategy of reducing the general abelian case to the jacobian case.

In the present paper, we axiomatize this strategy of \textit{reducing to the jacobian case} to enable us to use it in other contexts. A classical dimension argument shows that when the dimension $g$ is big, jacobian varieties become rarer in the moduli space of principally polarized abelian varieties, at least over $\C$.  

Over $\overline{\mathbb{Q}}$, Tsimerman \cite{Tsi12} proved that for every $g\geq  4$ there exist abelian varieties of dimension $g$ that are not isogenous to any jacobian. Masser and Zannier \cite{MaZa20} even recently proved the following: 
for every $g\geq  4$ ``almost every'' principally polarized abelian variety of dimension
$g$, defined over an extension of $\Q$ of degree at most $2^{16g^4}$, is Hodge generic and not
isogenous to any jacobian (see Theorem 1.3 in \cite{MaZa20} for the precise statement).

Hence, reducing to the case of jacobians is \textit{a priori} non-trivial. The jacobians we reduce to may have significantly larger dimension than the original abelian variety. The main advantage of this reduction lies in the fact that more tools are available for jacobians than for general abelian varieties.
\\

Before stating our first result we need to introduce the Northcott number of a set of algebraic numbers, following the terminology of \cite{ViVi16} page 59. 
Here we use $h_\infty(\cdot)$ for the absolute logarithmic Weil height on $\Qbar$ (defined in Section \ref{defs}), whereas in \cite{ViVi16} the exponential  Weil height is used,
so that their Northcott number is the exponential of the Northcott number defined here.
\begin{defin}\label{Northcott}
Let $S$ be a set of algebraic numbers. For any real number $t$, define the set $S_t=\{s\in{S}\,\vert\, h_\infty(s)\leq t\}$. Let $\m(S)=\inf\{t\,\vert\, S_t \, \mathrm{is}\; \mathrm{infinite}\}$. We will call $\m(S)$ the \emph{Northcott number} of $S$. 
\end{defin}

In particular we have $\m(\Qbar)=0$ and $\m(K)=\infty$ for any number field $K$. 
\begin{convention}\label{convention}
For the remainder of the introduction  we fix  an arbitrary infinite subset  $S$ of the algebraic numbers $\Qbar$ with finite Northcott number $\m(S)$.
\end{convention}
For a projective variety $Y$, let $\hp(Y)$ stand for the height of a Chow form of $Y$ (this is a projective height defined in Section \ref{defs}). In this paper varieties are always assumed to be geometrically irreducible.
Our first result is a Bertini-type statement with some control on the genus, the degree, the height, and the field of definition of the resulting curve. 

\begin{thm}\label{theo1}
Let $K$ be a subfield of $\overline{\mathbb{Q}}$, and let $X$ be a non-singular closed subvariety of $\mathbb{P}^N_{K}$ with $\dim X\geq2$. 
Then there exists a finite subset $s\subset S$, and a non-singular, geometrically irreducible curve $C$ on $X$, defined over $K(s)$, with genus $g(C)\leq (\deg X)^2+\deg X$,
with $\deg C\leq \deg X$ and with $\hp(C)\leq \hp(X)+(\dim X)(\deg X)(N+1)(\m(S)+2)$.
\end{thm}

We will now apply Theorem \ref{theo1} to the case of abelian varieties and derive some consequences. 
All number fields are assumed to lie in a fixed algebraic closure  $\overline{\mathbb{Q}}$.
For each abelian variety $A$ whose (minimal) field of definition is a number field $K$, we consider a  real-valued map $q(A,\cdot)$ with domain
the set of all  finite extensions $L$ of $K$. For example,  one could take $q(A, L)$
to be the Mordell-Weil rank of the group of rational points $A(L)$ or the dimension of $A$. 

We write $$\Qf:=\{q(A,\cdot)\,\vert\, A \text{ is an abelian variety defined over a number field}\}$$ for the family of these maps $q(A,\cdot)$.

In this article $\N=\{0,1,2,3,\ldots\}$ denotes the set of non-negative integers, $\R^+$ the set of positive real numbers, and $\R^+_0$ the set of non-negative real numbers.
All functions $c_i(\cdot)$ are real valued and non-negative.

\begin{defin}\label{admissible}
Given  maps $\czero, \cone, \ctwo: \mathbb{N}\to \mathbb{R}^+_0$ with $\ctwo$ increasing (not necessarily strictly), we say that the family of maps $\Qf$ is $(S, \czero, \cone, \ctwo)$-admissible if it satisfies the following:

For every number field $K$ and every abelian variety $A/K$ of  dimension $g\geq2$, and for every finite subset $s\subset S$, for every abelian variety $B/K(s)$, the following three properties hold:
\begin{description}
\item[(E)]  $q(A,K(s))\geq \czero(g) q(A,K)$,
\item[(P)]  $q(A\times B,K(s))\geq \cone(g) q(A,K(s)) $,
\item[\;(I)] $\vert q(A,K(s))-q(B,K(s))\vert\leq \ctwo(g)$ whenever $A$ and $B$ are $K(s)$-isogenous.

\end{description}
\end{defin}

Note that the homomorphisms between abelian varieties of dimension $g$ and defined over $K$, can be defined over a finite extension of $K$ with degree bounded above in terms of $g$ only, which is true by \cite{Sil92}, see also \cite{Rem20} for a recent optimized bound.

The following result shows that certain estimates for abelian varieties using the stable Faltings height $\hF(A)$ can be reduced to the case of jacobians.

\begin{thm} \label{machine}
Let $\czero,\cone,\ctwo,\constii,\constiii, \cthree,\cfour:\N\to \R^+_0$, and assume $\czero,\cone,\constiii, \cthree$ are positive.
Assume that the family of maps $\Qf$ is $(S, \czero, \cone, \ctwo)-$admissible.
Let $K$ be a number field, $A/K$ an abelian variety defined over $K$ of dimension $g\geq 2$, and suppose there exists
a finite set $s\subset S$ and a non-singular, geometrically irreducible curve $C\subset A$, defined over $K(s)$, and of genus $g_0$, such that
\begin{description}
\item[(i)] there exists a closed immersion $A\to \Jac(C)$, defined over the field of definition of $C\subset A$,
\item[(ii)] $g_0\leq \constii(g)$,
\item[(iii)] $\hF(\Jac(C))\leq \constiii(g)(\hF(A)+m(S)+1)$.
\end{description}
Then, if 
$$\hF(\Jac(C))\geq \cthree(g) q(\Jac(C), K(s)) - \cfour(g)$$
holds, we also have
$$\hF(A)\geq \cfive(g) q(A,K) - \csix(g),$$
where $\cfive(g)=\frac{\czero(g)\cone(g)\cthree(g)}{\constiii(g)}$ and $\csix(g)=\frac{\ctwo(\constii(g))\cthree(g)}{\constiii(g)}+\frac{\cfour(g)}{\constiii(g)}+\m(S)+1$. 
\end{thm}

Theorem \ref{theo1}, in conjunction with work of Cadoret and Tamagawa, allows us to prove the existence of a curve $C$ with the required properties,
provided $A$ is principally polarized. We also need an additional technical assumption which can always be achieved by extending the ground field.

Let $K$ be a number field, and let $A$ be an abelian variety of dimension $g$ defined over $K$. We say $A/K$ satisfies the condition $(\AK)$ if:
\\
  
$(\AK)$\begin{tabular}{ll}
&$A/K$ is semi-stable and equipped with a projective embedding $\Theta: A\to \mathbb{P}_K^{4^{2g}-1}$,\\
      & corresponding to the $16$th power of a symmetric theta divisor, with a theta\\ 
      & structure of level $r=4$, given by the modified theta coordinates of \cite{DavPhi} \\
      &  section 3.1.
\end{tabular}
\\

For any positive integer $N$, we denote by $A[N]$ the set of all $N$-torsion points in $A(\overline{\mathbb{Q}})$. We will then denote by $K(A[N])$ the smallest field extension $F/K$ such that $A[N]\subset A(F)$. If $A/K$ is principally polarized then we can ensure $(\AK)$ by replacing the ground field  $K$ with $K(A[48])$. 
Indeed, an embedding $\Theta$ always exists (based on Mumford's construction, see for instance paragraph 2.3 in \cite{Paz12}) after replacing the base field $K$ with the extension $K(A[16])$. 
Moreover, $A$ is semi-stable over the field $K(A[N])$ whenever $N$ has at least two coprime divisors $d_1, d_2 \geq 3$ (by Raynaud's criterion from \cite{SGA7}, Proposition 4.7. See also \cite{SiZa95} for extensions of this result). Note also that for any positive integer $N$ we have 
$[K(A[N]):K] \leq N^{4g^2}$ by Lemme 4.7 page 2078 of \cite{GaRe2}.

\begin{Proposition} \label{Cexists}
There exists a map $\constiii:\N\to \R^+$ such that the following holds.
If $K$ is a number field, and $A/K$ is a principally polarized abelian variety defined over $K$ of dimension $g\geq 2$ satisfying $(\AK)$, 
then there exists a finite set $s\subset S$ and a non-singular, geometrically irreducible curve $C\subset A$ defined over $K(s)$, and of genus $g_0$, such that
\begin{description}
\item[(i)] there exists a closed immersion $A\to \Jac(C)$, defined over the field of definition of $C\subset A$,
\item[(ii)] $g_0\leq (16^gg!)^2+16^gg!$,
\item[(iii)] $\hF(\Jac(C))\leq \constiii(g)(\hF(A)+m(S)+1)$.
\end{description}
\end{Proposition}

Note that the curve $C$ is not necessarily semi-stable over $K(s)$, and that $\Jac(C)$ does not come automatically with a theta structure of level 4 over $K(s)$. This is why the item $(iii)$ concerns stable Faltings heights, and not heights over $K(s)$.\\

\begin{cor} \label{machinecorrollary}
Let $\czero, \cone, \ctwo, \cthree, \cfour:\N\to \R^+_0$, with $\czero,\cone,\cthree$ positive, and suppose  the family of maps $\Qf$ is $(S, \czero, \cone, \ctwo)-$admissible.
Set $\constii(g)=(16^gg!)^2+16^gg!$.
Then there exists a map $\constiii:\N\to \R^+$ such that the following holds:

If $K$ is a number field and the inequality
\begin{description}
\item[(J)] \hspace{2.6cm} $\hF(J)\geq \cthree(g) q(J, K(s)) - \cfour(g)$
\end{description}
holds for all finite $s\subset S$, for all jacobians $J/K(s)$ of dimension $g_0\leq \constii(g)$, then 
$$\hF(A)\geq \cfive(g)q(A,K) - \csix(g)$$ for all principally polarized abelian varieties $A/K$ of dimension $g\geq 2$ satisfying $(\AK)$, where $\cfive(g)$ and $\csix(g)$ are as in Theorem \ref{machine}.
\end{cor}

Sometimes Zarhin's trick can be used to get rid of the assumption that $A$ be principally polarized and  the property $(\AK)$. Let us denote the dual of $A$ by $\check{A}$.
We set 
$$Z(A)=A^4\times \check{A}^4.$$
\begin{cor} \label{machinecorrollaryallabvar}
Let $\czero, \cone, \ctwo, \cthree, \cfour:\N\to \R^+_0$, with $\czero,\cone,\cthree$ positive, and suppose  the family of maps $\Qf$ is $(S, \czero, \cone, \ctwo)-$admissible.
Set $\constii(g)=(16^gg!)^2+16^gg!$.
Then there exists a map $\constiii:\N\to \R^+$ such that, with $\cfive(g)$ and $\csix(g)$ as in Theorem \ref{machine}, the following holds:
Let $A/K$ be an abelian variety of dimension $g$. 

i) If $Z(A)/K$ satisfies $(\AK)$, and  the  hypothesis (J) (with $8g$ instead of $g$) holds true,  then we have
$$\hF(A)\geq \frac{\cfive(8g)\cone(g)}{8} q(A,K) - \frac{\csix(8g)}{8}.$$ 

ii) If  the hypothesis (J) from Corollary \ref{machinecorrollary}  (with $8g$ instead of $g$) holds for all extensions $L$ with   $[L:K] \leq 48^{256 g^2}$  instead of just $K$,
then we have
$$\hF(A)\geq \frac{\cfive(8g)\cone(g)}{8} q(A,L_0) - \frac{\csix(8g)}{8},$$ 
for some  extension $L_0$ of $K$ with $[L_0:K] \leq 48^{256 g^2}$, even if $Z(A)/K$ does not satisfy $(\AK)$.
\end{cor}
To deduce Corollary \ref{machinecorrollaryallabvar} we first note that $Z(A)$ admits a principal polarization. For the first claim 
we  combine Theorem \ref{machine} (with $\constii(g)=(16^gg!)^2+16^gg!$ and $\constiii(\cdot)$ as in Proposition \ref{Cexists}) and Proposition \ref{Cexists} with $A$ replaced by $Z(A)$. Using that $\dim Z(A)=8g$, $\hF(Z(A))=8 \hF(A)$, and that by (P) we have $q(Z(A), K)\geq \cone(g)q(A,K)$ the claim follows.
For the second part we use that by the discussion before Proposition \ref{Cexists} there exists an extension $L_0$ of $K$ such that $Z(A)/L_0$ satisfies $(\AK)$ and $[L_0:K] \leq 48^{256 g^2}$.
We replace $K$ by  $L_0$ and conclude as before. 

Let us illustrate Corollary \ref{machinecorrollaryallabvar}  by a result of the first author  \cite{Paz16} that pioneered the strategy used in this paper. 
We take 
\begin{equation*}
q(A,K)=\frac{1}{[K:\mathbb{Q}]}\log \left(N_{K/\mathbb{Q}}\Big(\prod_{\mathfrak{p} \, \mathrm{bad}\, s.s.}\mathfrak{p}\Big)\right),
\end{equation*}
where the product runs over the semi-stable
bad prime ideals of $A/K$. 
It is clear that $(P)$ holds with $\cone(g)=1$. 
Isogenous abelian varieties share the same semi-stable bad reduction primes by the N{\'e}ron-Ogg-Shafarevich
criterion, because they have the same Tate modules (see Theorem 1 page 493 of \cite{SeTa68} and
Corollary 2 page 22 of \cite{Fal86}). Hence, (I) holds with $\ctwo(g)=0$. For (E) we need to assume that
our fixed $S$ from Convention \ref{convention} is such that $K(S)/K$ has ramification uniformly bounded above all finite prime ideals of $K$ with a bound independent of $K$
(cf. Proposition \ref{NN2Trem} to construct such $S$).
If $M/K$ is a finite extension then $A/K$ has semi-stable bad reduction at $\mathfrak{p}\subset \Oseen_K$ if and only if  $A/M$ has semi-stable bad reduction at $\mathfrak{B}\subset \Oseen_M$
for each prime $\mathfrak{B}$ above $\mathfrak{p}$. Hence, with an $S$ as above, a straightforward computation shows that (E) holds.
Finally, as shown in \cite{Paz16} hypothesis (J) follows essentially from the 
arithmetic Noether's formula of \cite{MB} Th\'{e}or\`{e}me 2.5 page 496. By Corollary  \ref{machinecorrollaryallabvar} (ii)
 it follows that there exist
$\cfive:\N\to \R^+_0$ and $\csix:\N\to \R^+$ such that
\begin{equation}\label{Faltbad}
\hF(A)\geq \frac{\cfive(8g)}{8\cdot 48^{256g^2}[K:\mathbb{Q}]}\log \left(N_{K/\mathbb{Q}}\Big(\prod_{\mathfrak{p} \, \mathrm{bad}\, s.s.}\mathfrak{p}\Big)\right) -\frac{\csix(8g)}{8}.
\end{equation}

Combining (\ref{Faltbad}) with an additional argument (Lemma 3.5 of \cite{Paz16}) shows that one can take the product in (\ref{Faltbad}) even over all primes with bad reduction, 
but at the expense of  replacing the stable Faltings height $\hF(A)$ on the left hand-side with the relative Faltings height $\hF(A/K)$.  This was used in \cite{Paz16} to establish  (\ref{heightrank}).

Next let us consider another consequence of Theorem \ref{machine}. Taking $q(A,K)=\mathrm{rank}(A(K))$ we see that $\Qf$ is $(S,1,1,0)-$admissible.

Given an abelian variety $A/K$ of dimension $g\geq 1$ then  $q(A,L)$ is unbounded as $L$ runs over all finite extensions of $K$.
Let $\constii(\cdot), \constiii(\cdot)$ be given, and choose $\czero(g)=1,\cone(g)=1,\ctwo(g)=0,\cthree(g)=1$ and $\cfour(g)=0$.
Take an extension $L/K$ such that  $\hF(A)< \cfive(g) q(A,L) - \csix(g)$. 
Then apply Theorem \ref{machine} with $A/K$ replaced by $A/L$ to deduce 
the following corollary.

\begin{cor}\label{rank intro}
Let $A$ be an abelian variety over $\overline{\mathbb{Q}}$ of dimension $g\geq2$, let $\constii, \constiii:\N\to \R^+$.
Then there exists a number field $L$ such that $A$ is defined over $L$, and
for every finite set $s\subset S$, and for every non-singular, geometrically irreducible curve $C\subset A$ defined over $L(s)$, and satisfying $(i), (ii), (iii)$ from Theorem \ref{machine} (with the given maps $\constii(\cdot), 
 \constiii(\cdot)$),
we have $$\hF(\Jac(C))< \mathrm{rank}(\Jac(C)(L(s))).$$ 
\end{cor}

Let us recall that Honda conjectured in \cite{Honda} page 98 that for any abelian variety $A/K$ there exists a constant $c_A>0$ such that if $M$ is an extension of  the number field $K$ 
then
$$\mathrm{rank}(A(M))\leq c_A [M:\mathbb{Q}].$$

As a consequence of Proposition \ref{Cexists} we can reduce Honda's conjecture to the case (of a strong form) of jacobians.

\begin{cor}\label{Honda}
Let $A$ be a principally polarized abelian variety over $K$ of dimension $g\geq2$ satisfying $(\AK)$. 
Suppose that there exists a map $\ctwenty: \N\times \R\to \R^+$, increasing in both variables, and  such that  
$$\mathrm{rank}(\Jac(C)(M))\leq \ctwenty(g_0,\hF(\Jac(C))) [M:\mathbb{Q}],$$ for all finite subsets $s\subset S$, for all non-singular, geometrically irreducible curves $C\subset A$, defined over $K(s)$, 
and of genus $g_0\leq (16^gg!)^2+16^gg!$, and for all finite extensions $M/K(s)$.
Then there exist a finite subset $s_A\subset S$, and a map $\ctwentyseven:\N\times \R\times \R^+_0\to \R^+$ such that for each finite extension $M/K(s_A)$
 $$\mathrm{rank}(A(M))\leq \ctwentyseven(g,\hF(A),m(S))[M:\mathbb{Q}].$$
\end{cor}

Using again Zarhin's trick, we conclude that if $\mathrm{rank}(J(M))\leq \ctwenty(g_0,\hF(J)) [M:\mathbb{Q}]$ for all finite subsets $s\subset S$, all jacobians
$J/L(s)$, for each extension $L/K$ of degree at most $48^{256 g^2}$, and of genus $g_0\leq  (16^{8g}(8g)!)^2+16^{8g}(8g)!$, and all finite extensions $M/L(s)$, then there exists a finite subset $s_A\subset S$, and a map $\ctwentyseven:\N\times \R\times \R^+_0\to \R^+$ such that for each finite extension $M/K(s_A)$
 $$\mathrm{rank}(A(M))\leq 48^{256 g^2}\ctwentyseven(8g,8\hF(A),m(S))[M:\mathbb{Q}].$$
In particular, if the strong form of Honda's conjecture (with $c_A=\ctwenty(g,\hF(A))$ as above) holds true for all jacobians,
then Honda's conjecture holds true for all abelian varieties over $\Qbar$.

A different contribution related to Honda's conjecture is in Pasten's recent paper \cite{Pas19}.

We define the main tools in Section \ref{defs}. As Theorem \ref{theo1} is expressed using the Philippon height of Chow forms, and as we would like to obtain corollaries involving the Faltings height, we also gather what is needed to translate the inequalities into this other height theory. We prove Theorem \ref{theo1} in Section \ref{Bertinisec} (for background on Bertini theorems, a good reference is \cite{Jou83}). 

As usual we proceed by intersecting the projective variety $X$ with hyperplanes but we use the fact that all their coefficients can be assumed to lie in our fixed set $S$ with finite Northcott number. 
To control the height and the degree we  use previous work of R\'emond \cite{Remond}. The control on the genus is obtained through previous work of Cadoret and Tamagawa \cite{CaTa12}.
In Section \ref{sectionProposition} we explain how to apply Theorem \ref{theo1} to abelian varieties. 
In Section \ref{exmachina} we prove Theorem \ref{machine}, and in Section \ref{ProofHonda} we prove Corollary \ref{Honda}.
Finally, in Section \ref{North}, we provide some methods to construct infinite sets with finite Northcott numbers. In particular, we show how to construct an 
infinite set $S\subset \Qbar$ with finite Northcott number such that $K(S)/K$ has uniformly bounded ramification for every number field $K$, and this bound is also uniform in $K$.

The authors thank Philipp Habegger and Ga\"el R\'emond for their helpful feedback on an earlier version of the text, and Hector Pasten for helpful comments
on Honda's conjecture. It is our pleasure to thank the referee for a very dedicated work and for helpful feedback, that led to improvements in several places.

The first author is supported by ANR-17-CE40-0012 Flair and ANR-20-CE40-0003-01 Jinvariant.

\section{Definitions and height comparisons}\label{defs}

\subsection{Height of algebraic numbers and projective points} Let $K$ be a number field and $M_K$ the set of places of $K$. We denote by $M_{K}^{0}$ the set of finite places and $M_K^{\infty}$ the set of archimedean places. For any prime number $p>0$, we normalize the archimedean absolute values by $\vert p\vert_v=p$ and the non-archimedean by $\vert p\vert_v=p^{-1}$ if $v$ divides $p$.
 We denote by $K_v$ the completion of $K$ with respect to $\vert \cdot\vert_v$. Let $d=[K:\mathbb{Q}]$. For any place $v$ of $K$, let $d_v=[K_v:\mathbb{Q}_v]$. 

Let $N\geq 1$ be an integer and $x=(x_0,\ldots ,x_N)$ a vector of algebraic numbers in $K$, not all zero. Let $\displaystyle{\Vert x\Vert_v=\max_{0\leq i\leq N}\vert x_i\vert_v}$ for $v$ a non-archimedean place of $K$ and $\displaystyle{\Vert x\Vert_v=\Big(\sum_{0\leq i\leq N} \vert x_i\vert_v^2\Big)^{1/2}}$ for $v$ an archimedean place of $K$.

For $P=(x_0:\cdots:x_N)\in\mathbb{P}^{N}_{\Qbar}$ 
we define the 
$l_2$-height and the $l_\infty$-height (or Weil height) as
\begin{alignat*}1
h_2(P)&=\sum_{v\in{M_K}}\frac{d_v}{d}\log \Vert x\Vert_v,\\
h_{\infty}(P)&=\sum_{v\in{M_K}}\frac{d_v}{d}\log \max_{0\leq i\leq N}\vert x_i\vert_v,
\end{alignat*}
where $K$ is any number field containing the coordinates $x_0,\ldots,x_N$. Dividing by the degree $d=[K:\Q]$
makes the value independent of the particular choice of $K$, and thanks to the product formula these heights
are well-defined on $\mathbb{P}^{N}_{\Qbar}$.
Moreover, they are comparable; we have 
\begin{equation}\label{compar}
h_{\infty}(P)\leq h_2(P) \leq h_{\infty}(P)+\frac{1}{2}\log(N+1).
\end{equation}
Besides the height of a projective point we also need to measure the height of an algebraic number $x$. By abuse of notation
we write 
\begin{alignat*}1
h_2(x)&=h_2(P),\\
h_{\infty}(x)&=h_\infty(P),
\end{alignat*}
where $P=(1:x)\in \mathbb{P}^{1}_{\Qbar}$.

\subsection{Height of a polynomial with algebraic coefficients} 
If $P=x_0X^N+\cdots +x_{N-1}X+x_N\in \Qbar[X]$ is a non-zero polynomial 
we define $h_2(P)=h_2(x_0:\cdots:x_N)$. More generally, if  
$$P=\sum_{i_1=0}^{N}\cdots \sum_{i_n=0}^{N}x_{{i_1}\cdots{i_n}}X_1^{i_1}\cdots X_n^{i_n}\in 
\Qbar[X_1,\ldots,X_n]\backslash\{0\}$$
then we define $$h_2(P)=h_2(\cdots: x_{{i_1}\cdots{i_n}} :\cdots).$$ 
Analogously, we define 
$$h_\infty(P)=h_\infty(\cdots: x_{{i_1}\cdots{i_n}} :\cdots).$$

\subsection{Height of a Chow form}\label{Chow}
Let us now consider $X$ a geometrically irreducible projective variety inside $\mathbb{P}^N$ defined over the number field $K$. We follow \cite{Remond} to define the height of $X$. Let $F$ be its Chow form. We define the height of $X$ by $$h_{\mathbb{P}^N}(X)=\sum_{v\in{M_K^{0}}}\frac{d_v}{d}\log \Vert F\Vert_v +\sum_{v\in{M_K^{\infty}}}\frac{d_v}{d}\Big[(\dim X +1)(\deg X)\sum_{j=1}^N \frac{1}{2j}+\int_{S_{N+1}^{\dim X +1}}\log \vert F_v\vert \tau\Big], $$ where $\tau$ is the invariant measure of total mass 1 on $S_{N+1}^{\dim X +1}$, the $(\dim X +1)$-power of the unit sphere $S_{N+1}$, and in the expression $\Vert F\Vert_v$ we identify $F$ with the vector of its coefficients. Again, this definition is independent of the choice of $K$.

When $X$ is a general closed subscheme of $\mathbb{P}^N$ defined over a number field, we define its height as the sum of the previously defined heights of its irreducible components. 
 We note that $h_{\mathbb{P}^N}(\cdot)$ is non-negative (see for instance \cite{Phi3} paragraph 1 page 346).

\subsection{Height of an abelian variety}

For the special case of abelian varieties, we will mostly use the Faltings height, and in some estimates also the theta height. We recall their definition and give a brief summary of some useful properties and comparisons of these heights.

\subsubsection{Faltings height}
Let $A$ be an abelian variety of dimension $g\geq1$ defined over a number field $K$. Let ${\mathcal O}_K$ be the ring of integers of $K$ and let $\pi\colon {\mathcal A}\longrightarrow \Spec(\mathcal{O}_K) $
be the N\'eron model of $A$ over $\Spec(\mathcal{O}_K)$. Let $\varepsilon\colon \Spec(\mathcal{O}_K)\longrightarrow {\mathcal A}$ be the zero section of $\pi$ and let
$\omega_{{\mathcal A}/\mathcal{O}_K}$ be the pullback along $\varepsilon$ of the maximal exterior power (the determinant) of the sheaf of relative differentials
$$\omega_{{\mathcal A}/\mathcal{O}_K}:=\varepsilon^{\star}\Omega^g_{{\mathcal
A}/\mathcal{O}_K}\;.$$

\noindent For any archimedean place $v$ of $K$, let $\sigma$ be an embedding of $K$ in $\mathbb{C}$ associated to $v$. The associated line bundle
$$\omega_{{\mathcal A}/\mathcal{O}_K,\sigma}=\omega_{{\mathcal A}/\mathcal{O}_K}\otimes_{{\mathcal O}_K,\sigma}\mathbb{C}\simeq H^0({\mathcal
A}_{\sigma}(\mathbb{C}),\Omega^g_{{\mathcal A}_\sigma}(\mathbb{C}))\;$$
is equipped with the $L^2$-metric $\Vert.\Vert_{v}$ given by
$$\Vert s\Vert_{v}^2=\frac{i^{g^2}}{\gamma^{g}}\int_{{\mathcal
A}_{\sigma}(\mathbb{C})}s\wedge\overline{s}\;$$
where $\gamma>0$ is a normalizing constant. In this article we choose  $\gamma=(2\pi)^2$.

The projective ${\mathcal O}_K$-module $\omega_{{\mathcal A}/\mathcal{O}_K}$ is of rank $1$ and together with the hermitian norms
$\Vert.\Vert_{v}$ at infinity it defines a hermitian line bundle 
$\overline{\omega}_{{\mathcal A}/\mathcal{O}_K}=({\omega}_{{\mathcal A}/\mathcal{O}_K}, (\Vert .\Vert_v)_{v\in{M_K^\infty}})$ over $\mathcal{O}_K$. It has a well defined Arakelov degree
$\widehat{\degr}(\overline{\omega}_{{\mathcal A}/\mathcal{O}_K})$, given by 
$$\widehat{\degr}(\overline{\omega}_{{\mathcal A}/\mathcal{O}_K})=\log\#\left({\omega}_{{\mathcal A}/\mathcal{O}_K}/{s{\mathcal
O}}_K\right)-\sum_{v\in{M_{K}^{\infty}}}d_v\log\Vert
s\Vert_{v}\;,$$
where $s$ is any non-zero section of ${\omega}_{{\mathcal A}/\mathcal{O}_K}$. The resulting number does not depend on the choice
of $s$ in view of the product formula on the number field $K$. 

The height of $A/K$ is defined as
$$\hF(A/K):=\frac{1}{[K:\mathbb{Q}]}\widehat{\degr}(\overline{\omega}_{{\mathcal
A}/\mathcal{O}_K})\;.$$
It does not depend on any choice of projective embedding of $A$. 
We emphasise that by our choice of normalization  $\gamma=(2\pi)^2$ the Faltings height $\hF(\cdot)$ is non-negative (see Remarque 3.3 in \cite{Paz19} and a detailed proof in the appendix of \cite{GaRe3} for the semi-stable case, the general case follows by property (2) below).
Faltings \cite{Falt} used the normalization  $\gamma=2$ so that $\hF(A/K)=h_{Faltings}(A/K)+\frac{g}{2}\log(2\pi^2)$.

We recall here three classical properties (see for instance \cite{Del}, in particular page 35) used in the sequel:

\begin{enumerate}
\item If $A=A_1\times A_2$ is a product of abelian varieties over $K$ then one has $\hF(A/K)=\hF(A_1/K)+\hF(A_2/K)$.
\item If $K'/K$ is a number field extension then one has $\hF(A/K')\leq \hF(A/K)$.
\item If $A/K$ is semi-stable then the height is stable by number field extension.
\end{enumerate}

\begin{defin} \label{faltings} 
The stable height of $A/K$ is defined as $\hF(A):=\hF(A/K')$ for any number field extension $K'/K$ such that $A/K'$ is semi-stable.
\end{defin}

\subsubsection{Theta height}
We refer the reader to classical work of Mumford on theta structures, recalled in some details in paragraph 2.3 of \cite{Paz12}. Let $A$ be an abelian variety over a number field $K$. 
Assume $A/K$ is given with a theta structure of level $4$. It gives in particular an explicit embedding of $K$-varieties $\Theta: A\to \mathbb{P}^{4^{2g}-1}_K$, and we define $$h_{\Theta}^{(4)}(A,\Theta):= h_2(\Theta(0_A)).$$
We note that $h_{\Theta}^{(4)}(\cdot,\cdot)$ is also a non-negative height.

\subsubsection{Useful inequalities}
We gather here some technical inequalities between these different heights. Let us start by considering $C$ a non-singular curve in $\mathbb{P}_{\Qbar}^N$ with $N\geq4$, of degree $\deg C$ and genus $g$. 
Choose a ground field $K$ such that $\Jac(C)/K$ has $(\AK)$.
By Th\'eor\`eme 1.3 and Proposition 1.1 page 760 of \cite{Remond} one has
\begin{equation}\label{remlem}
h_{\Theta}^{(4)}(\Jac(C),\Theta)\leq (2\deg C+1)^2\log(N+1)^4 \, m^{20m8^{g}}(\hp(C)+1),
\end{equation}
where $m=4g+2\deg C -2$. 

As we will also need a comparison between the Faltings height and the theta height of level $r=4$ of a principally polarized 
 abelian variety $A/K$ of dimension $g$ with $(\AK)$,
 we extract the following Bost-David comparison from \cite{Paz12}, Corollary 1.3 page 21 (due to the different normalization $\gamma=2\pi$  our height here differs by 
$g\log(2\pi)/2$ from the one in \cite{Paz12}) using the factor $7$ instead of $6$ to take care of the different normalization:

\begin{equation*}
\vert h_{\Theta}^{(4)}(A,\Theta)-\frac{1}{2} h_{F}(A)\vert \leq 7 \cdot 4^{2g}\cdot  \log(4^{2g}) \cdot\log\Big(\max\{1, h_{\Theta}^{(4)}(A,\Theta)\}+2\Big).
\end{equation*}
It follows that there exists an explicitly computable map $\cthirtytwo:\N\to \R^+$ such that
\begin{equation}\label{thetafaltings}
h_{\Theta}^{(4)}(A,\Theta)-\cthirtyone(g)\leq h_{F}(A)\leq 3h_{\Theta}^{(4)}(A,\Theta)+\cthirtyone(g).
\end{equation}

Finally, we need a Philippon height -- Faltings height comparison lemma.

\begin{lem}\label{chowfal}
Let $A$ be a principally polarized abelian variety of dimension $g$ over $\overline{\mathbb{Q}}$ given with a projective embedding $\Theta: A\to \mathbb{P}^{4^{2g}-1}$ compatible with a theta structure of level $r=4$, corresponding to the 16th power of a symmetric theta divisor. There is an explicitly computable map $\cseven:\N\to \R^+$ such that $$\hp(A)\leq \cseven(g) (\hF(A)+1),$$ where $N=16^g-1$.
\end{lem}

\begin{proof}
We use Proposition 3.9 of \cite{DavPhi} page 665, where the authors prove that for any algebraic subvariety $V\subset A$ the inequality 
$$\vert \widehat{h}_{\mathbb{P}^N}(V)-h_{\mathbb{P}^N}(V)\vert \leq \ceight(g, \dim V, \deg V, h_{\Theta}^{(4)}(A,\Theta))$$
holds, where ${h}_{\mathbb{P}^N}(V)$ is the height of the variety $V$ as defined previously in paragraph \ref{Chow} (it is the same definition as the one in \cite{DavPhi} page 644), the height $\widehat{h}_{\mathbb{P}^N}(V)$ is defined in \cite{Phi} before Proposition 9 and the quantity $\ceight(g, \dim V, \deg V, h_{\Theta}^{(4)}(A,\Theta))>0$ can be taken to be $(4^{g+1}h_{\Theta}^{(4)}(A,\Theta)+3g\log2)\cdot(\dim V+1)\cdot\deg V$. Picking $V=A$ the abelian variety we focus on, we have $\widehat{h}_{\mathbb{P}^N}(A)=0$ (see Proposition 9 item (vii) page 281 of \cite{Phi}), $\dim A=g$, $\deg A = 16^g g!$ (see \cite{Mum} page 150) and 
\begin{equation}\label{h1}
h_{\mathbb{P}^N}(A)\leq \cnine(g) (h_{\Theta}^{(4)}(A,\Theta)+1),
\end{equation}
 where $\cnine(g)>0$ only depends on the dimension of $A$, and one can take $\cnine(g)=4^{3g+1}(g!) (g+1)$. 
Plugging the estimate from (\ref{thetafaltings}) into (\ref{h1}) and using that $\hF(A)\geq 0$,  concludes the proof.
\end{proof}

\section{Bertini with height control}\label{Bertinisec}

As a first step, we state a classical Bertini Theorem, then we use a result of R\'emond to control the height of a curve drawn on a projective variety. 

\begin{thm}\label{Bertini}(Bertini's Theorem)
Let $X$ be a non-singular closed subvariety of $\mathbb{P}^N_{\overline{\mathbb{Q}}}$ with $\dim X\geq2$. There exists a hyperplane $H_0\subset \mathbb{P}^N_{\overline{\mathbb{Q}}}$ not containing $X$ and such that $X\cap H_0$ is non-singular, geometrically irreducible of dimension $\dim X -1$. Furthermore, the set of such hyperplanes forms an open dense subset $U$ of the complete linear system $\vert H_0\vert$, viewed as a projective space.
\end{thm}
\begin{proof}
This is a particular case of Theorem II.8.18 of \cite{Hart} page 179, see also Remark 7.9.1 of \cite{Hart} page 245.
\end{proof}
\begin{cor}\label{control}
Let $X$ be as in Theorem \ref{Bertini}, and let $S$ be an infinite set of algebraic numbers. There exists a hyperplane $H_0$ defined with coefficients in $S$ such that $X\cap H_0$ is non-singular, geometrically irreducible, of dimension $\dim X -1$.
\end{cor}
\begin{proof}
Let $\check{\mathbb{P}}^{N}$ be the dual projective space. Consider the isomorphism  $j:\mathbb{P}^N\longrightarrow \check{\mathbb{P}}^{N}$ defined by $j([s_0:\cdots : s_N])=H_{s_0,\ldots, s_N}$, where $H_{s_0,\ldots, s_N}$ is the hyperplane defined by $s_0x_0+\cdots+s_Nx_N=0$.  We look at the set $U$ obtained in Theorem \ref{Bertini}.
The set $j^{-1}(U)$ is non-empty and open in $\mathbb{P}^{N}$.  
Let $P\in \Qbar[x_0,\ldots,x_N]$ be a non-zero homogeneous polynomial. Since $S$ is infinite there exists $[s_0:\cdots:s_N]\in \mathbb{P}^N$ with all $s_i\in S$ and $P(s_0,\ldots, s_N)\neq 0$.
Hence, $$j^{-1}(U)\cap \{[s_0:\cdots:s_N]\in{\mathbb{P}^N} \,\vert\, s_0, \ldots, s_N\in{S}\;\, \textrm{not}\; \textrm{all}\; \textrm{zero}\}\neq \emptyset,$$ 
and this implies the claim.
\end{proof}

The following result of R\'emond is the main tool for Theorem \ref{theo1}. It is a direct consequence of Proposition 2.3 page 765 of \cite{Remond}.
\begin{prop}(R\'emond)\label{Remond}
Let $X$ be a closed subscheme of $\mathbb{P}^N_{\overline{\mathbb{Q}}}$. Let $P_1, \ldots, P_s$ be homogeneous polynomials of $\overline{\mathbb{Q}}[X_0,\ldots, X_N]$ of degree at most $D$ and of height $h_2(\cdot)$ at most $H$. If $\mathcal{V}$ is the family of irreducible components of the intersection $Y$ of $X$ with the zeros of $P_1,\ldots, P_s$, then $$\sum_{V\in{\mathcal{V}}}D^{\dim V}\deg V \leq D^{\dim X} \deg X,$$ and if one denotes $d=\min\{\dim V \vert V\in{\mathcal{V}}\}$, $$\sum_{V\in{\mathcal{V}}}D^{\dim V +1}\hp(V)\leq D^{\dim X +1} \hp(X) + (\dim X - d) D^{\dim X} (\deg X) H.$$ In particular, $\deg Y\leq D^{\dim X-d}\deg X$ and $$\hp(Y)\leq D^{\dim X -d} \hp(X) +(\dim X- d)D^{\dim X -d -1}(\deg X) H.$$
\end{prop}

R\'emond's original version is slightly stronger, as he only requires a modified height to be bounded by $H$. The height used in the inequalities, however, is the same as the one used in the present work.

\begin{prop}\label{control2}
Let $S$ be an infinite set of algebraic numbers with finite Northcott number $\m(S)$. Let $X$ be a non-singular closed subvariety of $\mathbb{P}^N_{\overline{\mathbb{Q}}}$ with $\dim X\geq2$. Then there exists a hyperplane $H_0$ defined with coefficients in $S$ and such that $X\cap H_0$ is non-singular, geometrically irreducible with dimension $\dim X -1$, with $\deg(X\cap H_0)\leq \deg X$, and $$\hp(X\cap H_0) \leq \hp(X) +(\deg X) (N+1)(\m(S)+2).$$
\end{prop}

\begin{proof}
First let us replace $S$ with its infinite subset of elements $s$ with $h_\infty(s)<\m(S)+1$.
By Corollary \ref{control}
we get a hyperplane $H_0$, defined by a non-zero linear form  $F_0$,
with coefficients  $s_0,\ldots,s_N \in S$. 
Thus,
$$h_\infty(F_0)=h_\infty(s_0:\cdots:s_N)\leq h_\infty(1:s_0:\cdots:s_N)\leq \sum_{i=0}^N h_\infty(1:s_i)\leq (N+1)(\m(S)+1).$$
Finally, using that $h_2(F_0)\leq h_\infty(F_0)+\frac{1}{2}\log(N+1)\leq (N+1)(\m(S)+2)$, and applying
Proposition \ref{Remond} with $H=(N+1)(\m(S)+2)$, $d=\dim X-1$, and $D=1$ yields the claim.
\end{proof}
We are now in position to prove Theorem \ref{theo1}. We will prove the following, slightly more precise, result.
\begin{cor}\label{curves} (Bertini with height control) 
Let $S$ be an infinite set of algebraic numbers with finite Northcott number $\m(S)$. Let $X$ be a non-singular closed subvariety of $\mathbb{P}^N_{\overline{\mathbb{Q}}}$ with $\dim X\geq2$. Then there exists a non-singular, geometrically irreducible curve $C$ on $X$, defined over a finite extension of the field of definition of $X$ by finitely many elements of $S$, with $\deg C\leq \deg X$, and 
$$\hp(C) \leq \hp(X) +(\dim X)(\deg X) (N+1)(\m(S)+2).$$ Moreover, the genus of $C$ may be assumed to be bounded from above by $(\deg X)^2+\deg X$ 
and if $X=A$ is a principally polarized abelian variety, one may assume in addition that there is a closed immersion $A\to \Jac(C)$ over the field of definition of $C$.
\end{cor}

\begin{proof}
We apply Proposition \ref{control2} to the successive intersections $X\cap H_1\cap\cdots \cap H_i$, where $i\geq1$ is an integer. We reach dimension $1$ in $\dim X-1$ steps, the curve $C$ will be $X\cap H_1\cap\ldots\cap H_{g-1}$ if $g=\dim X$. The control on the genus of $C$ is based on Castelnuovo's criterion of \cite{ACGH} page 116 (see also  Remark 2.1 \cite{CaTa12}). 

In the case where $X$ is a principally polarized abelian variety $A$ and both $A$ and $C$ are defined over an infinite field $K$ of characteristic zero, the closed immersion $A\to \Jac(C)$ is obtained from studying the fundamental groups of the successive intersections (independently of the choice of hyperplanes). The closed immersion $A\cap H_1\to A$ induces a surjective morphism between \'etale fundamental groups $\pi_1(A\cap H_1)\to \pi_1(A)$. Iterating $g-1$ times, the closed immersion $C\to A$ induces a surjective homomorphism $\pi_1(C)\to \pi_1(A)$ over $\mathrm{Gal}(\overline{K}/K)$.

This implies that the Albanese morphism $\Jac(C)\to A$ is surjective with connected kernel. Hence the dual morphism $\check{A}\to\check{\Jac(C)}$ is a closed immersion. As both $A$ and $\Jac(C)$ are principally polarized, we have a closed immersion $A\to \Jac(C)$.

For more details see Lemme X.2.10 of \cite{SGA}, recalled in Lemma 4.1 of \cite{CaTa12}, and the arguments detailed in the last section of \cite{CaTa12}. The only difference with Theorem 1.2 of \cite{CaTa12} is that the hyperplanes we chose have bounded height.
\end{proof}

\section{Theorem \ref{theo1}  applied to abelian varieties}\label{sectionProposition}

We turn to an application of Theorem \ref{theo1} to abelian varieties.
In particular, we prove Proposition \ref{Cexists}.

\begin{lem}\label{tech}
Suppose $N\geq 4$. There exists a map $\cten: \N^3\to \R^+$ such that the following holds.
For any non-singular  geometrically irreducible curve $C$ in $\mathbb{P}^N$ of genus $g_0$, 
we have
\begin{equation}\label{h2}
\hF(\Jac(C))\leq \cten(g_0,\deg C,N)(h_{\mathbb{P}^N}(C)+1).
\end{equation}
\end{lem}

\begin{proof}
After an extension of the base field we can assume that $\Jac(C)$ satisfies $(\AK)$. We conclude from inequality (\ref{remlem}) that
$h_{\Theta}^{(4)}(\Jac(C),\Theta)\leq \ctwentyeight(g_0,\deg C,N)(\hp(C)+1)$ where $\ctwentyeight(g_0,\deg C,N)>0$.
Using (\ref{thetafaltings}) with $A=\Jac(C)$, and the fact that $\hp(C)\geq 0$ yields the claim.
\end{proof}

Next we prove Proposition \ref{Cexists}. For the convenience of the reader we recall the statement. 
\begin{Proposition}\label{corona}
Let $S\subset \Qbar$ be an infinite set  with finite Northcott number $m(S)$.
Then there exists a map $\constiii:\N\to \R^+$ such that the following holds.
If $K$ is a number field, and $A/K$ is a principally polarized abelian variety defined over $K$ of dimension $g\geq 2$, and satisfying $(\AK)$, 
then there exists a finite set $s\subset S$ and a non-singular, geometrically irreducible curve $C\subset A$ defined over $K(s)$, and of genus $g_0$, such that
\begin{description}
\item[(i)] there exists a closed immersion $A\to \Jac(C)$, defined over the field of definition of $C\subset A$,
\item[(ii)] $g_0\leq (16^gg!)^2+16^gg!$,
\item[(iii)] $\hF(\Jac(C))\leq \constiii(g)(\hF(A)+m(S)+1)$.
\end{description}
\end{Proposition}

\begin{proof}

 By assumption we can embed our abelian variety $A$  via $\Theta: A\to \mathbb{P}_K^{N}$ (compatible with a theta structure of level $r=4$, corresponding to the 16th power of a symmetric theta divisor),
where $N=4^{2g}-1$.
By Corollary \ref{curves}  there exists a non-singular geometrically irreducible curve $C$ on $A\subset \mathbb{P}^{N}$ defined over a finite extension $K'/K$ obtained by adjoining elements of $S$, and such that 
\begin{equation}\label{g_0}
\hp(C) \leq \hp(A) +g(\deg A)(N+1)(\m(S)+2),
\end{equation}
$\deg C\leq \deg A$, and with the genus $g_0$ of $C$ bounded from above by $(\deg A)^2+\deg A$, and such that there is a closed immersion $A\to \Jac(C)$, defined over the field of definition of $C\subset A$. 
Since $\deg A=16^g g!$, we get the first two properties. Next we prove (iii).
Because $A$ is principally polarized, and by Lemma \ref{chowfal}, one has 
\begin{equation}\label{chowfaltings}
\hp(A)\leq \cseven(g)(\hF(A)+1).
\end{equation}

By Lemma \ref{tech} one has 
\begin{equation}\label{curve}
\hF(\Jac(C))\leq \cten(g_0,\deg C,N)(\hp(C)+1).
\end{equation}

Using successively (\ref{curve}), (\ref{g_0}), and (\ref{chowfaltings}), this gives  
\begin{equation*}
\hF(\Jac(C))\leq \cthirteen(g,g_0,\deg C,N)\hF(A)+\cfourteen(g,g_0,\m(S), \deg A, \deg C, N),
\end{equation*}
with the quantities 
\begin{alignat*}1
\cthirteen(g,g_0,\deg C,N)&=\cseven(g)\cten(g_0,\deg C,N),\\
 \cfourteen(g,g_0,\m(S),\deg A,\deg C,N)&=\cten(g_0,\deg C,N)\Big(\cseven(g)+1\\
 &\phantom{=}\quad g(\deg A) (N+1)(\m(S)+2))\Big).
 \end{alignat*}
Finally, note that $g_0\leq (16^gg!)^2+16^gg!$, $\deg C\leq \deg A=16^g g!$, and $N=16^g-1$, and since $\cten(\cdot,\cdot,\cdot)$ can be assumed to be increasing in each variable,
we get the required map $\constiii:\N\to \R^+$ as claimed.
\end{proof}

\section{Proof of Theorem \ref{machine}}\label{exmachina}

Let us start by the following lemma.
\begin{lem}\label{lemme clef}
Let $A$ be an abelian variety of dimension $g$ over a number field $K$, and let $C\subset A$
be a non-singular geometrically irreducible algebraic curve of genus $g_0$ over a number field extension $K'$ of $K$ such that there exists a closed immersion $A\to \Jac(C)$ defined 
over $K'$. Then there exists 
an abelian variety $B$ defined over $K'$ such that $\Jac(C)$ is $K'$-isogenous to $A\times B$.
\end{lem}
\begin{proof}
By abuse of notation, we may view $A$,  via the closed immersion, as an abelian subvariety of $\Jac(C)$ defined over $K'$.
From  Poincar\'e's Reducibility Theorem we get that  $\Jac(C)$ is $K'$-isogenous to $A\times B$, where $B$ is the quotient of $\Jac(C)$ by the image of $A$.
\end{proof}

We are now ready to prove Theorem \ref{machine}.

\begin{proof}
By hypothesis there exists
a finite set $s\subset S$ and a non-singular, geometrically irreducible curve $C\subset A$, defined over $K(s)$, and of genus $g_0$, such that
\begin{description}
\item[(i)] there exists a closed immersion $A\to \Jac(C)$ defined over the field of definition of $C\subset A$,
\item[(ii)] $g_0\leq \constii(g)$,
\item[(iii)] $\hF(\Jac(C))\leq \constiii(g)(\hF(A)+m(S)+1)$.
\end{description}
By Lemma \ref{lemme clef}, we know that $\Jac(C)$ is $K(s)$-isogenous to $A\times B$ for some abelian
variety $B$ over $K(s)$. 
By hypothesis we have
\begin{equation}\label{HF-ineq}
\hF(\Jac(C))\geq \cthree(g) q(\Jac(C),K(s))-\cfour(g).
\end{equation}
 
Next note that  $\dim(\Jac(C))\geq \dim A\geq 2$, and hence by property (I)

\begin{equation*}
q(\Jac(C),K(s))\geq q(A\times B,K(s))-\ctwo(g_0),
\end{equation*}
and by property (P)
\begin{equation*}
q(A\times B,K(s))\geq \cone(g)q(A,K(s)),
\end{equation*}
and finally by property (E)
\begin{equation*}
q(A,K(s))\geq \czero(g)q(A,K).
\end{equation*}
Plugging the last three inequalities into (\ref{HF-ineq}) yields
\begin{equation*}
\hF(\Jac(C))\geq \czero(g)\cone(g)\cthree(g) q(A,K)-\ctwo(g_0)\cthree(g)-\cfour(g).
\end{equation*}

Finally, we apply hypothesis $(iii)$ and use that $\ctwo(\cdot)$ is increasing and $g_0\leq \constii(g)$ to conclude
\begin{equation*}
\hF(A)\geq \frac{\czero(g)\cone(g)\cthree(g)}{\constiii(g)} q(A,K)-\frac{\ctwo(\constii(g))\cthree(g)}{\constiii(g)}-\frac{\cfour(g)}{\constiii(g)}-\m(S)-1.
\end{equation*}
\end{proof}

\section{Proof of Corollary \ref{Honda}}\label{ProofHonda}
We now turn to the proof of Corollary \ref{Honda}.

\begin{proof}

By Proposition \ref{Cexists} there exists a finite subset $s_A\subset S$ and a  non-singular geometrically irreducible curve $C\subset A$,
defined over $K(s_A)$, and of genus $g_0$, satisfying (i), (ii) and (iii). By  Lemma \ref{lemme clef}  we have that $\Jac(C)$ is $K(s_A)$-isogenous to $A\times B$
for some abelian variety $B$ defined over $K(s_A)$.
Let $M/K(s_A)$ be a finite extension. Then, with $\constii(g)=(16^gg!)^2+16^gg!$, we have
\begin{alignat*}2
\mathrm{rank}(A(M))&\leq \mathrm{rank}((A\times B)(M))&\\
&= \mathrm{rank}(\Jac(C)(M))&\\
&\leq \ctwenty(g_0,\hF(\Jac(C)))[M:\mathbb{Q}] & (\text{by hypothesis})\\
&\leq \ctwenty(\constii(g),\hF(\Jac(C)))[M:\mathbb{Q}] &(\text{by (ii)}) \\
&\leq \ctwenty(\constii(g),\constiii(g)(\hF(A)+m(S)+1))[M:\mathbb{Q}]  &(\text{by (iii)}).
\end{alignat*}
\end{proof}

\section{Sets and field extensions with finite Northcott number}\label{North}

Recall Definition \ref{Northcott} of the Northcott number. 
In some applications of Theorem \ref{theo1} and Theorem \ref{machine} it is essential that the curve $C$ is defined
over an extension $K(s)/K$ with uniformly bounded (or otherwise prescribed) ramification, with a bound independent of $K$. Hence, we need to construct an infinite set $S\subset \Qbar$
with finite Northcott number and such that $K(S)/K$ has uniformly bounded ramification, i.e., the ramification indices
$e(\mathfrak{B}/\mathfrak{B}\cap K)$ are uniformly bounded as $\mathfrak{B}$ runs over the finite prime ideals of  all 
finite extensions of $K$ contained in $K(S)$. It is also required that this bound is independent of $K$. 
Here we show that such a set $S$ exists. 
We also recall two other methods to produce infinite sets with finite Northcott number. 

\begin{lem}\label{intgen}
Let $L$ be a number field of degree $d$.
Then there exists an integral number $\alpha$ in $L$ with $L=\mathbb{Q}(\alpha)$ and
\begin{alignat*}1
h_\infty(\alpha)\leq \frac{1}{d}\log|\Delta_L|.
\end{alignat*}
\end{lem}
\begin{proof}
We make use of the exponential height given by $H(x)=\exp(h_\infty(x))$ for any algebraic number $x$. 
If $L$ has a real embedding then the claim follows from \cite{VaWi13} Theorem 1.2 since the proof of this result 
yields an algebraic integer.

Suppose the field $L$ has no real embeddings.
Let $\sigma_{1},\overline{\sigma_{1}},\ldots,\sigma_{s},\overline{\sigma_{s}}$ be the
$s$ pairs of complex conjugate embeddings of $L$. With $\sigma=(\sigma_1,\ldots,\sigma_{s})$ we get an embedding
$\sigma:L\longrightarrow  \mathbb{C}^s.$
The image of the ring of integers $\mathcal{O}_L$ is a lattice with determinant $2^{-s}\sqrt{|\Delta_L|}$. 
We define a convex set $S_T$ in $\mathbb{C}^s$, symmetric with respect to the origin, by
\begin{alignat*}1
S_T=\{(x_1,\ldots,x_{s})\in \mathbb{C}^s\,\vert\,&|\Im(x_1)|\leq T, |\Re(x_1)|<1,|x_2|<1,\ldots,|x_{s}|<1\}.
\end{alignat*}
So $S_T$ has volume $4\pi^{s-1}T$. By Minkowski's convex body theorem we conclude that $S_T$ contains a non-zero lattice point
whenever $T>(\pi/4)(2/\pi)^{s}\sqrt{|\Delta_L|}$.
But such a lattice point $\sigma(\alpha)$ must come from a primitive point $\alpha$, otherwise there exists an embedding $\sigma'$, different from $\sigma_1(\cdot)$,
with $\sigma_1(\alpha)=\sigma'(\alpha)$. If $\sigma'(\cdot)=\overline{\sigma_1}(\cdot)$ then  $\Im(\sigma_1(\alpha))=0$, and if $\sigma'(\cdot)\neq \overline{\sigma_1}(\cdot)$
then $|\sigma_1(\alpha)|=|\sigma'(\alpha)|<1$. In both cases we conclude that  $|\sigma_1(\alpha)\cdots \sigma_{s}(\alpha)|<1$, and hence $\alpha=0$.
As $\alpha$ is integral we get $H(\alpha)\leq (T^2+1)^{1/d}$,
and since $|\Delta_L|\geq 2$ we may conclude  $H(\alpha)\leq |\Delta_L|^{1/d}$. Taking the logarithm yields the claimed inequality.
\end{proof}

Since  $|\Delta_L|=|\Delta_K|^{[L:K]}$ whenever $L/K$ is a finite unramified  extension
one can conclude from the previous lemma  that if $F/K$ is an infinite unramified extension then $\m(F)\leq\frac{1}{[K:\Q]}\log\vert \Delta_{K}\vert $.  
In general it is not so easy to decide whether a number field $K$ has an infinite unramified extension (cf.  \cite{Mai00}).
However, the following proposition shows that there exists an infinite set $S\subset \Qbar$ with finite Northcott number such that $K(S)/K$ has uniformly bounded ramification, with a bound independent of $K$.

 \begin{prop}\label{NN2Trem}
There exists an infinite set $S\subset \Qbar$ with finite Northcott number, and $E_S\in \N$ such that for every number field $K$ and every finite extension $L/K$ with $L\subset K(S)$,
we have  $e(\mathfrak{B}/\mathfrak{B}\cap K)\leq E_S$ for every prime ideal $\mathfrak{B}$ in $\Oseen_L$.
 \end{prop}
 \begin{proof}
By, e.g., \cite{Mar78} there exists a number field  $K_0$ that admits an infinite unramified extension $F/K_0$. Then $KF/KK_0$ is infinite and unramified (see Proposition B.2.4 page 592 of \cite{BoGu07}). Hence, if $L/K$ is finite and $L\subset KF$ 
then for any prime ideal $\mathfrak{Q}$ in $\Oseen_{LK_0}$ we have 
$e(\mathfrak{Q}/\mathfrak{Q}\cap K)=e(\mathfrak{Q}/\mathfrak{Q}\cap  KK_0)e(\mathfrak{Q}\cap  KK_0/\mathfrak{Q}\cap K)= e(\mathfrak{Q}\cap KK_0/\mathfrak{Q}\cap K)\leq [KK_0:K]\leq [K_0:\Q]$. In particular,  $e(\mathfrak{B}/\mathfrak{B}\cap K)\leq [K_0:\Q]$ for every prime ideal $\mathfrak{B}$ in $\Oseen_L$.
Hence, with $S=F$ and $E_S=[K_0:\Q]$ the extension $K(S)/K$ has ramification uniformly bounded by $E_S$, and we  know from  the previous observation that $\m(S)\leq \frac{1}{[K_0:\Q]}\log\vert \Delta_{K_0}\vert$.
\end{proof}

 The problem of finding $K_0$  that admits an infinite unramified extension and minimises $\frac{1}{[K_0:\Q]}\log\vert \Delta_{K_0}\vert$ has found much interest, see for instance
Martinet \cite{Mar78} who showed that $K_0=\Q(\cos(2\pi/11),\sqrt{2},\sqrt{-23})$ admits an infinite unramified $2$-tower and satisfies $$\frac{1}{[K_0:\Q]}\log\vert \Delta_{K_0}\vert\leq 4.53.$$

If $K=\Q$ and if we are allowed to have ramification only above a single rational prime $p$, then we can take 
$S=\{p^{1/p^i}\,\vert\, i\in \N\}$ so that
 
$\m(S)=0$. However,  in this example the ramification is not uniformly bounded.

A different way of describing essentially the same example is to take $S=\{x_i\,\vert\, i\in \N\}$ with $x_0=p$ and $x_i\in \Qbar$ satisfying 
$P(x_{i+1},x_i)=0$ for $P(x,t)=x^p-t$.
  This approach of constructing sets with finite Northcott number can be generalised as follows.
 \begin{lem}\label{Height quasi-equiv}
Let $P(x,t)\in \Qbar[x,t]$ be irreducible with $\deg_x P>\deg_t P>0$, and let $S=\{x_i\,\vert\, i\in \N\}$ with $x_i$ pairwise distinct algebraic numbers
satisfying $P(x_{i+1},x_i)=0$ for $i\in \N$.
Then 
$$\m(S)\leq \deg_t P\left(\frac{\gamma_P\deg_x P}{\deg_x P-\deg_t P}\right)^2,$$

where 

$\gamma_P = 5(\log(2^{\min\{\deg_x P, \deg_t P\}}(\deg_x P + 1)(\deg_t P + 1))+h_\infty(P))^{1/2}$.
\end{lem}
\begin{proof}
This follows from an explicit version  of a result of N\'eron \cite{Ne65}
$$\left|\frac{h_\infty(x_{i+1})}{\deg_t P}-\frac{h_\infty(x_i)}{\deg_x P} \right|\leq \gamma_P\max\left \{\frac{h_\infty(x_{i+1})}{\deg_t P},\frac{h_\infty(x_i)}{\deg_x P}\right \}^{1/2}$$ 
due to Habegger \cite[Theorem 1]{Hab17}. Habegger's inequality implies  
$$h_\infty(x_{i+1})\leq qh_\infty(x_{i})+Q\max\{h_\infty(x_{i}),h_\infty(x_{i+1})\}^{1/2},$$ 
where
$Q=\gamma_P\sqrt{\deg_t P}$, $q=\frac{\deg_t P}{\deg_x P}$, and the upper bound in  Lemma \ref{Height quasi-equiv} is just $\left(\frac{Q}{1-q}\right)^2$.
We leave the details to the reader.
\end{proof}

For the polynomial $P(x,t)=x^2-tx-1$, Smyth (\cite[Theorem 1, pages 137-138]{Smy80})\label{Smyth}  proved much more.
Indeed,  set $x_0=1$, and 
suppose $S=\{x_i\}_{i=0}^\infty\subset \Qbar$ satisfy
$P(x_{i+1},x_i)=0$ for $i\in \N$. 
Then $x_i$ has degree $2^i$ over $\Q$, each $x_i$ is totally real, and the sequence of logarithmic Weil heights $h_\infty(x_i)$ has a limit point  $0.2732\ldots$.
In particular, $\m(S)\leq 0.274$.

Using the the well-known identity
$$h_\infty(\alpha)=\frac{1}{\deg f}\int_{0}^1\log|f(e^{2\pi i t})|dt,$$ 
where $f\in \Z[x]$ is the minimum polynomial of $\alpha$, one has yet another method to construct infinite subsets $S\subset \Qbar$ with finite Northcott number.
Let $f=a_0x^d+\cdots +a_d \in \Z[x]$, and let us write $\|f\|_1=|a_0|+\cdots+|a_d|$ for the length of $f$.

Now let $\{f_i\}_{i=0}^\infty \subset \Z[x]$ be an infinite set of non-constant irreducible polynomials.
Let $S\subset \Qbar$ be such that $S$ contains a root for each polynomial $f_i$.
Then 
$$\m(S)\leq \liminf_{i}\frac{\log \|f_i\|_1}{\deg f_i}.$$

\begin{exa} 
Consider the polynomials  $f_i(x)=x^i-x-1\in \Z[x]$. Selmer \cite{Sel56} has shown that for $i>1$ they are all irreducible
(and Osada \cite[Corollary 3]{O87} showed that they have full Galois group $S_i$). Therefore, $m(S)=0$.
\end{exa}

\begin{exa} 
Consider a sequence of monic polynomials  $f_i(x)\in \Z[x]$ of degree $i$ whose constant term is equal to $\pm p_i$ for a prime $p_i$, and such that 
$\|f_i\|_1<2p_i$. These polynomials are all irreducible over $\Z$. Otherwise, $f_i=gh$ with $g,h\in \Z[x]$, and $g$ has constant term $\pm 1$. Hence $f_i$ would have a zero $\alpha$ of complex absolute value  at most $1$, 
and thus  $p_i=|\alpha^i+a_1\alpha^{i-1}+\cdots+a_{i-1}\alpha|\leq |\|f_i\|_1-p_i|<p_i$.
If $\log p_i=o(i)$, we conclude $m(S)=0$.
\end{exa}


\begin{thebibliography}{widest-label}

\bibitem[ACGH85]{ACGH} \textsc{Arbarello, E., Cornalba, M., Griffiths, P.A. and Harris, J.}, 
\textit{Geometry of algebraic curves}.
G.M.W., Springer-Verlag, {\bf 267.1} (1985).

\bibitem[BoGu07]{BoGu07} \textsc{Bombieri, E. and Gubler, W.}, 
\textit{Heights in Diophantine Geometry}.
 New Mathematical Monographs, Cambridge University Press 2006 (reprint 2007), {\bf 4} (2007).
 
 
\bibitem[CaTa13]{CaTa12} \textsc{Cadoret, A. and Tamagawa, A.}, 
\textit{Note on the torsion conjecture}.
In "Groupes de Galois G\'eom\'etriques et differentiels", P. Boalch and J.-M. Couveignes eds., S\'eminaires et Congr\`es, S.M.F. {\bf 27} (2013), 57--68.

\bibitem[DaHi00]{DaHi00} \textsc{David, S. and Hindry, M.}, 
\textit{Minoration de la hauteur de N\'eron-Tate sur les vari\'et\'es abeliennes de type CM}.
J. reine angew. Math. {\bf529} (2000), 1--74.

\bibitem[DaPh02]{DavPhi} \textsc{David, S. and Philippon, P.}, 
\textit{Minorations des hauteurs normalis\'ees des sous--vari\'et\'es
              de vari\'et\'es abeliennes. {II}}.
Comment. Math. Helv. {\bf77} (2002), 639--700.

\bibitem[Del83]{Del} \textsc{Deligne, P.}, 
\textit{Preuves des conjectures de Tate et Shafarevitch (d'apr\`es G. Faltings)}.
S\'eminaire Bourbaki, 36e ann\'ee, 1983-1984, {\bf 616} (1983), 25--41.


\bibitem[Fal83]{Falt} \textsc{Faltings, G.}, 
\textit{Endlichkeitss\"atze f\"ur abelsche {V}ariet\"aten \"uber {Z}ahlk\"orpern}.
Invent. Math. {\bf73} (1983), 349--366.

\bibitem[Fal86]{Fal86} \textsc{Faltings, G.}, 
\textit{Finiteness theorems for abelian varieties over number fields}.
Arithmetic Geometry, Cornell and Silverman (editors), Springer-Verlag  (1986), 9--27.

 \bibitem[GaR\'e14]{GaRe2}  \textsc{Gaudron, E. and R\'emond, G.}, 
\textit{Polarisations et isog\'enies}.
Duke Math. J. {\bf 163} (2014), no. 11, 2057--2108.

 \bibitem[GaR\'e14b]{GaRe3}  \textsc{Gaudron, E. and R\'emond, G.}, 
\textit{Th\'eor\`eme des p\'eriodes et degr\'es minimaux d'isog\'enies}.
Comment. Math. Helv. {\bf 89} (2014), 343--403.


 \bibitem[SGA1]{SGA}  \textsc{Grothendieck, A. et al.}, 
\textit{Rev\^etements \'etales et groupe fondamental}.
S\'eminaire de G\'eom\'etrie Alg\'ebrique du Bois Marie (1960--1961, SGA 1) Lecture Notes in Mathematics {\bf 224} (1971).

\bibitem[SGA7]{SGA7}  \textsc{Grothendieck, A. et al.}, 
\textit{Mod\`eles de N\'eron et monodromie}, in Groupes de monodromie en g\'eom\'etrie alg\'ebrique
(SGA 7) Lecture Notes in Math. {\bf 288}, Springer, Berlin-Heidelberg-New York (1972), 313--523.

\bibitem[Hab17]{Hab17} \textsc{Habegger, P.}, 
\textit{Quasi-equivalence of heights and Runge's Theorem.} Number Theory -- Diophantine problems, uniform distribution and applications, Springer, Cham (2017), 257--280.

\bibitem[Har06]{Hart} \textsc{Hartshorne, R.}, 
\textit{Algebraic Geometry.} Graduate Texts in Mathematics, Springer-Verlag {\bf 52} (2006).

\bibitem[HiSi00]{HiSi00} \textsc{Hindry, M. and Silverman, J.}, 
\textit{Diophantine Geometry: an introduction.} Graduate Texts in Mathematics, Springer-Verlag {\bf 201} (2000).

\bibitem[Hon60]{Honda} \textsc{Honda, T.}, 
\textit{Isogenies, rational points and section points of group varieties.} Japanese Journal of Mathematics, {\bf 30} (1960), 84--101.

\bibitem[Jou83]{Jou83} \textsc{Jouanolou, J.-P.}, 
\textit{Th\'eor\`emes de Bertini et applications.} Progress in Mathematics, {\bf 42} (1983), Birk\"auser Boston Inc., Boston, MA, (French).


\bibitem[Mai00]{Mai00} \textsc{Maire, C.}, 
\textit{On infinite unramified extensions.} Pacific J. of Math. {\bf 192.1} (2000), 135--142.

\bibitem[Mar78]{Mar78} \textsc{Martinet, J.}, 
\textit{Tours de corps de classes et estimations de discriminants.} Invent. Math. {\bf 44.1} (1978), 65--73.


\bibitem[MaZa20]{MaZa20} \textsc{Masser, D. and Zannier, U.},
\textit{Abelian varieties isogenous to no jacobian.} Ann. of Math. (2) {\bf 191} (2020), 635--674.

\bibitem[MB89]{MB} \textsc{Moret-Bailly, L.}, 
\textit{La formule de Noether pour les surfaces arithm\'etiques.} Invent. Math. {\bf 98} (1989), 491--498.

\bibitem[Mum70]{Mum} \textsc{Mumford, D.}, 
\textit{Abelian varieties.} Oxford University Press, first edition, (1970).

\bibitem[Ner65]{Ne65} \textsc{N\'eron, A.}, 
\textit{Quasi-fonctions et hauteurs sur les vari\'et\'es ab\'eliennes.} Ann. of Math. (2) {\bf 82} (1965), 249--331.

\bibitem[Osa87]{O87} \textsc{Osada, H.}, 
\textit{The Galois Groups of the Polynomials $X^n + aX^l+ b$.} J. Number Theory {\bf 25} (1987), 230--238.

\bibitem[Pas19]{Pas19} \textsc{Pasten, H.}, 
\textit{Bounded ranks and Diophantine error terms}. Math. Research Lett. {\bf 26} (2019), 1559--1570.

\bibitem[Paz12]{Paz12} \textsc{Pazuki, F.}, 
\textit{Theta height and Faltings height}. Bull. Soc. Math. France {\bf 140.1} (2012), 19--49.

\bibitem[Paz19a]{Paz19} \textsc{Pazuki, F.}, 
\textit{D\'ecomposition en hauteurs locales}. Contemp. Math.
 {\bf 722} (2019), 121--140.

\bibitem[Paz19b]{Paz16} \textsc{Pazuki, F.}, 
\textit{Heights, ranks and regulators of abelian varieties}. Ramanujan Math. Soc., Lecture Notes Series, to appear,
https://arxiv.org/abs/1506.05165{\bf}.

\bibitem[Phi91]{Phi} \textsc{Philippon, P.}, 
\textit{Sur les hauteurs alternatives I}. Math. Annalen {\bf 289} (1991), 255--283.

\bibitem[Phi95]{Phi3} \textsc{Philippon, P.}, 
\textit{Sur les hauteurs alternatives III}. J. Math. Pures Appl. (9) 74 (1995), no. 4, 345--365.

\bibitem[R\'em10]{Remond} \textsc{R\'emond, G.}, 
\textit{Nombre de points rationnels des courbes.} Proc. Lond. Math. Soc. {\bf 101.3} (2010), 759--794.

%
\bibitem[R\'em20]{Rem20} \textsc{R\'emond, G.}, 
\textit{Degr\'es de d\'efinition des endomorphismes d'une vari\'et\'e ab\'elienne} J. Eur. Math. Soc. {\bf 22} (2020), 3059--3099.


\bibitem[Sel56]{Sel56} \textsc{Selmer, E.S.}, 
\textit{On the irreducibility of certain trinomials.} Math. Scand. {\bf 4} (1956), 287--302.

\bibitem[SeTa68]{SeTa68} \textsc{Serre,J. P. and Tate J.}, 
\textit{Good reduction of abelian varieties.} Annals of Math. (2) {\bf 88} (1968), 492--517.

\bibitem[Sil92]{Sil92} \textsc{Silverberg, A.}, 
\textit{Fields of definition for homomorphisms of abelian varieties.} J. of Pure and Applied Algebra {\bf 77} (1992), 253--262.

 \bibitem[SiZa95]{SiZa95} \textsc{Silverberg, A. and Zarhin, Yu.},
\textit{Semistable reduction and torsion subgroups of abelian varieties}.
Ann. Inst. Fourier {\bf45} (1995), 403--420.

\bibitem[Smy80]{Smy80} \textsc{Smyth, C.J.}, 
\textit{On the measure of totally real algebraic integers.} J. Austral. Math. Soc. (Series A) {\bf 30} (1980), 137--149.

\bibitem[Tat66]{Tate66} \textsc{Tate, J.}, 
\textit{Endomorphisms of abelian varieties over finite fields.} Invent. Math. {\bf 2} (1966), 134--144.

\bibitem[Tsi12]{Tsi12} \textsc{Tsimerman, J.}, 
\textit{The existence of an abelian Variety over $\Qbar$ isogenous to no Jacobian.} Ann. of Math. (2) {\bf 176} (2012), 637--650. 

\bibitem[VaWi13]{VaWi13} \textsc{Vaaler, J.D. and Widmer, M.},
\textit{A note on small generators of number fields}.
Diophantine methods, lattices and arithmetic theory of quadratic forms, Contemp. Math. {\bf587} (2013), 201--211.

\bibitem[ViVi16]{ViVi16} \textsc{Vidaux, X. and Videla, C.R.},
\textit{A note on the Northcott property and undecidability}.
Bull. London Math. Soc. {\bf 48} (2016), 58--62.

%



\end{thebibliography}
\end{document}